\documentclass[11 pt, reqno]{amsart}

\usepackage{amssymb}
\usepackage{eucal}
\usepackage{amsrefs}

\theoremstyle{plain}
\newtheorem{theorem}{Theorem}
\newtheorem{proposition}[theorem]{Proposition}

\newtheorem{lemma}[theorem]{Lemma}

\theoremstyle{remark}
\newtheorem*{ackowledgements}{Ackowledgements}
\newtheorem*{notations}{Notations and conventions}
\newtheorem{remark}[theorem]{Remark}

\theoremstyle{definition}
\newtheorem{definition}[theorem]{Definition}

\numberwithin{equation}{section}

\DeclareMathOperator{\ad}{ad}
\DeclareMathOperator{\Aut}{Aut}
\DeclareMathOperator{\End}{End}

\DeclareMathOperator{\Lie}{Lie}
\DeclareMathOperator{\Res}{Res}
\DeclareMathOperator{\St}{St}

\newcommand{\bZ}{\mathbb Z}
\newcommand{\bQ}{\mathbb Q}
\newcommand{\bR}{\mathbb R}
\newcommand{\bC}{\mathbb C}

\newcommand{\sA}{\mathcal A}
\newcommand{\sB}{\mathcal B}
\newcommand{\sH}{\mathcal H}
\newcommand{\sV}{\mathcal V}

\newcommand{\sX}{\mathcal X}
\newcommand{\sY}{\mathcal Y}

\newcommand{\germg}{\mathfrak g}
\newcommand{\germk}{\mathfrak k}
\newcommand{\germp}{\mathfrak p}
\newcommand{\germS}{\mathfrak S}

\newcommand{\tG}{\widetilde G}
\newcommand{\tV}{\widetilde V}

\newcommand{\pel}{\textsc{pel}}

\begin{document}

\title{Hodge structures associated to $SU(p,1)$}

\author{Salman Abdulali}
\address{Department of Mathematics, East Carolina University, Greenville, NC 27858, USA}
\email{abdulalis@ecu.edu}



\begin{abstract}
Let $A$ be an abelian variety over $\bC$ such that the semisimple part of the Hodge group of $A$ is a product of copies of $SU(p,1)$ for some $p>1$.
We show that any effective Tate twist of a Hodge structure occurring in the cohomology of $A$ is isomorphic to a Hodge structure in the cohomology of some abelian variety. 
\end{abstract}

\maketitle

\section{Introduction}

This paper is a contribution to the classification of Hodge structures which occur in the cohomology of abelian varieties.
We consider the question of when an effective Tate twist of such a Hodge structure is itself isomorphic to a Hodge structure in the cohomology of some abelian variety.
Our interest in this question arises from the fact that the general Hodge conjecture, as formulated by Grothendieck, implies that any effective Tate twist of the Hodge structure occurs in the cohomology of some smooth, projective algebraic (not necessarily abelian) variety \cite{Grothendieck1969b}.

Recall that a (rational) Hodge structure of weight $n$ is a finite dimensional vector space $V$ over $\bQ$, together with a decomposition
\[
	V_{\bC} = \bigoplus_{p+q=n} V^{p,q}
\]
such that $V^{q,p} = \overline{V^{p,q}}$.
The Hodge structure is called \emph{effective} if $V^{p,q} = 0$ unless $p$,~$q \geq 0$.
For $m \in \bZ$, the \emph{Tate twist} $V(m)$ is the Hodge structure of weight $n-2m$ given by
$V(m)^{p,q} = V^{p+m,q+m}$.

In a series of papers
\cites{Abdulali1997, Abdulali1999, Abdulali2000, type3ghc,
Abdulali2004, Abdulali2005, Abdulali2011}
we have shown for a large class of abelian varieties that every effective Tate twist of a Hodge
structure contained in the cohomology of one of these abelian varieties is isomorphic to
a Hodge structure occurring in the cohomology of some abelian variety.
Our earlier results apply to abelian varieties of type IV in only a few cases---namely,
when the Hodge group is semisimple \cite{Abdulali1997}, or when
the abelian variety is of CM-type \cite{Abdulali2005}, or, when the semisimple part
of the Hodge group is a product of groups of the form $SU(p+1,p)$ \cite{Abdulali2011}.
In this paper we extend these results to any abelian variety $A$ such that the semisimple part of
the Hodge group of $A$ is a product of copies of $SU(p,1)$ for some $p>1$;
see Theorem~\ref{maintheorem} for the precise statement.

That not all abelian varieties have the above property was shown in \cite{type3ghc}.
In \cite{type3ghc}*{Theorem 5.5, p.~926} we showed the existence of a Hodge structure $M$ which occurs in
the cohomology of an abelian variety, such that $M(1)$ is effective but does not occur in
the cohomology of \emph{any} abelian variety.

We draw attention to two new features of this paper.
First, the concept of \emph{semidomination} (Definition~\ref{semidomination})
generalizes the concept of \emph{weak self-domination} introduced in \cite{Abdulali2011},
and allows us to work with the semisimple part of the Hodge group of an abelian variety
whose Hodge group is neither semisimple nor commutative (see Theorem~\ref{semi}).
Second, families of abelian varieties which are \emph{not} of \pel-type play a critical role,
even though our main concern is a general member of a \pel-family.

\begin{ackowledgements}
I am grateful to the referees for detailed and insightful comments which have led to significant improvements to this paper.
\end{ackowledgements}

\begin{notations}
All representations are finite-dimensional and algebraic.
The derived group of a group $G$ is denoted by $G'$.
All abelian varieties are over $\bC$.
For an abelian variety $A$, we let
\[
	D(A) = \End (A) \otimes \bQ
\]
be its endomorphism algebra, $L(A)$ its Lefschetz group, $G(A)$ its Hodge group,
$L'(A)$ the derived group of $L(A)$, and, $G'(A)$ the derived group of $G(A)$.
For a finite field extension $E$ of a field $F$, we let $\Res_{E/F}$ be
the restriction of scalars functor, from varieties over $E$ to varieties over~$F$.
The center of a group $G$ is denoted by $Z(G)$.
For a topological group $G$, we denote by $G^0$ the connected component of the identity.
\end{notations}

\section{Preliminaries}
\subsection{Kuga fiber varieties}

We briefly recall Kuga's construction of families of abelian varieties \cite{Kuga1};
full details may be found in Satake's book \cite{Satakebook}.

Let $G$ be a connected, semisimple, linear algebraic group over $\bQ$ with identity $1$.
Assume that $G$ is of hermitian type, and has no nontrivial, connected, normal $\bQ$-subgroup
$H$ with $H(\bR)$ compact.
Then $X = G(\bR)^0/K$ is a bounded symmetric domain for
a maximal compact subgroup $K$ of $G(\bR)^0$.
Let $\germg = \Lie G$ be the Lie algebra of $G$, $\germk = \Lie K$, and let
$\germg_{\bR} = \germk \oplus \germp$ be the corresponding Cartan decomposition.
Differentiating the natural map $\nu \colon G(\bR)^0 \to X$ induces an isomorphism of
$\germp$ with $T_o(X)$, the tangent space of $X$ at $o = \nu(1)$, and there exists
a unique $H_0 \in Z(\germk)$, called the \emph{$H$-element} at $o$,
such that $\ad H_0 | \germp$ is the complex structure on $T_o (X)$.

Let $\beta$ be a nondegenerate alternating form on a finite-dimensional rational vector space $V$.
The symplectic group $Sp(V,\beta)$ is a $\bQ$-algebraic group of hermitian type;
the associated symmetric domain is the \emph{Siegel space}
\begin{multline*}
	\germS(V,\beta) = \{\, J \in GL(V_{\bR}) \mid J^2 = -I \text{ and } \\
	\beta(x,Jy) \text{ is symmetric, positive definite}\, \}.
\end{multline*}
$Sp(V,\beta)$ acts on $\germS(V,\beta)$ by conjugation.
The $H$-element at $J \in \germS(V,\beta)$ is~$J/2$.

Let $\rho \colon G \to Sp(V, \beta)$ be a representation defined over $\bQ$.
We say that $\rho$ satisfies the \emph{$H_1$-condition} relative to the $H$-elements
$H_0$ and $H'_0 = J/2$ if
\begin{equation}
\label{h1}
	[d \rho (H_0) - H'_0, d \rho (g)] = 0 \qquad \text{for all }g \in \germg_{\bR}.
\end{equation}
If this is satisfied, then there exists a unique holomorphic map
$\tau \colon X \to \germS(V,\beta)$ such that $\tau (o) = J$, and the pair $(\rho, \tau)$
is equivariant in the sense that
\[
	\tau (g \cdot x) = \rho(g) \cdot \tau (x) \qquad \text{for all } g \in G(\bR)^0, \ x \in X.
\]
Let $\Gamma$ be a torsion-free arithmetic subgroup of $G(\bQ)$,
and $L$ a lattice in $V$ such that $\rho(\Gamma)L = L$.
The natural map
\[
	\sA = (\Gamma \ltimes_\rho L) \backslash (X \times V_{\bR})
	\longrightarrow \sV =  \Gamma \backslash X
\]
is a morphism of smooth quasiprojective algebraic varieties
(Borel \cite{Borel1972}*{Theorem 3.10, p.~559} and Deligne \cite{900}*{p.~74}),
so that $\sA$ is a fiber variety over~$\sV$ called a \emph{Kuga fiber variety}.
The fiber $\sA_P$ over any point $P \in \sV$ is an abelian variety isomorphic to the torus
$V_{\bR}/L$ with the complex structure $\tau (x)$, where $x$ is a point in $X$ lying over $P$.

\subsection{Two algebraic groups}
We recall  two algebraic groups associated to an abelian variety $A$ over $\bC$.
Let $V = H_1(A, \bQ)$, and let $\beta$ be an alternating Riemann form for $A$.
The \emph{Hodge group} $G(A)$, and \emph{Lefschetz group} $L(A)$ are
reductive $\bQ$-subgroups of $GL(V)$.
The Hodge group (or \emph{special Mumford-Tate group}) is characterized by the property that
for any positive integer~$k$, the invariants of the action of $G(A)$ on $H^{\bullet}(A^k,\bQ)$
form the ring $\sH(A^k)$ of Hodge classes (Mumford \cite{Mumford1966}).
The Lefschetz group is defined to be the centralizer of $\End (A)$ in $Sp(V, \beta)$;
it is characterized by the property that for any positive integer $k$,
the subring of $\sH(A^k)$ generated by the classes of divisors equals
$H^{\bullet}(A^k,\bC)^{L(A)_{\bC}} \cap H^{\bullet}(A^k,\bQ)$
(Milne \cite{Milne}*{Theorem~3.2, p.~656} and Murty \cite{murtybanff}*{\S 3.6.2, p.~93}).
Clearly,
\[
	G(A) \subset L(A) \subset Sp(V, \beta).
\]

The inclusion $G'(A) \hookrightarrow Sp(V, \beta)$ satisfies the $H_1$-condition \eqref{h1}
with respect to suitable $H$-elements.
Taking $L = H_1(A,\bZ)$, and $\Gamma$ to be any torsion-free arithmetic subgroup of $G'(A)$
such that $\gamma (x) \in L$ for all $\gamma \in \Gamma$, $x \in L$, we obtain a
Kuga fiber variety having $A$ as a fiber; this is called the \emph{Hodge family} determined by
$A$ (see Mumford \cite{Mumford1966}).

The inclusion $L'(A)^0 \hookrightarrow Sp(V, \beta)$ satisfies the $H_1$-condition,
and hence may be used to define a Kuga fiber variety having $A$ as a fiber.
These families of abelian varieties are generalizations of the \pel-families constructed by
Shimura in \cite{Shimura1963b}.

\begin{lemma}
\label{groupreduction}
Let $A$ be an abelian variety such that $G'(\bR)$ has no compact factors.
If $G'(A) = G_1 \times G_2 \times \dots \times G_t$, then $A$ is isogenous to a product
\[
	A \sim A_1 \times A_2 \times \dots \times A_t,
\]
where $G'(A_i) = G_i$ for $1 \leq i \leq t$.
\end{lemma}

\begin{proof}
Let $\sA \to \sV$ be the Hodge family of $A$, and let $P \in \sV$ be such that $A = \sA_P$.
Let $\rho \colon G(A) \to Sp(V, \beta)$ be the $H_1$-representation defining $\sA \to \sV$.
Assume without loss of generality that $\sV = \sV_1 \times \dots \times \sV_t$,
where $\sV_i$ is an arithmetic variety associated to $G_i$.
Let $P = (P_1, \dots, P_t)$ with $P_i \in \sV_i$.

We recall that a representation of $G$ is called \emph{primary} if all irreducible subrepresentations are equivalent.
Write $V = \bigoplus_{j=1}^s V_j$ where the $V_j$ are the maximal primary $G$-submodules of $V$,
and let
\[
	\rho_j \colon G \to Sp(V_j, \beta \vert V_j \times V_j)
\]
be the restriction of $\rho$ to $V_j$.
Then each $\rho_j$ satisfies the $H_1$-condition (see Satake \cite{Satakebook}*{Lemma IV.4.2, p.~181, and the remarks following it}).
By \cite{type3ghc}*{Lemma~2.4.1, p.~919}, each $\rho_j$ is nontrivial on $G_i$ for at most one $i$.
We may therefore write $\rho = \bigoplus_{i=1}^t \sigma_i$, where $\sigma_i$ is a representation of $G_i$ satisfying the $H_1$-condition.
Let $\sA_i \to \sV_i$ be the Kuga fiber variety determined by $\sigma_i$, and let $A_i$ be the fiber of $\sA_i$ over $P_i$.
Then $G'(A_i) = G_i$ and $A$ is isogenous to $A_1 \times A_2 \times \dots \times A_t$.
\end{proof}

\begin{proposition}
\label{center}
For any abelian variety $A$, the center of $G(A)$ is contained in the center of $L(A)$.
\end{proposition}

\begin{proof}
Let $V = H_1(A, \bQ)$, let $\beta$ be an alternating Riemann form for $A$, and,
let $Z$ be the center of $G(A)$.
Since $D(A)$ is the centralizer of $G(A)$ in $\End(V)$, it follows that $Z \subset D(A)$.
Since $L(A)$ is (by definition) the centralizer of $D(A)$ in $Sp(V,\beta)$, we see that
$Z$ is in the center of $L(A)$.
\end{proof}

\begin{remark}
\label{productremark}
If $A$ and $B$ are abelian varieties, then
\[
	H_1(A \times B, \bQ) = H_1(A, \bQ) \oplus H_1(B, \bQ)
\]
as Hodge structures,
so $G(A \times B)$ acts on the (co)homology of each factor.
It follows that $G(A \times B)$ is a subgroup of $G(A) \times G(B)$ whose projection to each factor is surjective (see Imai \cite{Imai}).
\end{remark}

\section{Semidomination}

We say that a Hodge structure $V_{\bC} = \bigoplus_{p+q=n} V^{p,q}$ of weight $n$ is
\emph{fully twisted} if it is effective, and, $V^{n,0} \neq 0$.
We say that a smooth, projective algebraic variety $A$ over $\bC$ is \emph{dominated} by
a set $\sX$ of smooth, projective algebraic varieties over $\bC$ if, for any irreducible Hodge
structure $V$ in the cohomology of $A$, there exists a fully twisted Hodge structure $V'$
in the cohomology of some $X \in \sX$ such that $V'$ is isomorphic to a Tate twist of $V$.
Grothendieck \cite{Grothendieck1969b}*{p.~301} observed that if $A$ is dominated by $\sX$, then the usual Hodge conjecture for  $A \times X$, for all $X \in \sX$, implies the general Hodge conjecture for $A$.

We now introduce the concept of \emph{semidomination} of an abelian variety, which will help us to prove that certain abelian varieties are dominated by classes of abelian varieties. Its main utility is in allowing us to work with the semisimple part of the Hodge group of an abelian variety whose Hodge group is neither semisimple nor commutative.

\begin{definition}
\label{semidomination}
We say that an abelian variety $A$ is \emph{semidominated} by a set~$\sX$ of abelian varieties if,
given any nontrivial irreducible representation~$\rho$ of $G'(A) _{\bC}$ such that $\rho$ occurs in
$H^n(A, \bC)$ for some $n$, there exist $A_{\rho} \in \sX$, a positive integer $c_{\rho}$, and,
$V_{\rho} \subset H^{c_{\rho}} (A_{\rho}, \bC)$, such that
\begin{itemize}
	\item
	$V_{\rho}$ is a $G(A_{\rho}) _{\bC}$-submodule of $H^{c_{\rho}} (A_{\rho}, \bC)$,
	\item
	the action of $G'(A \times A_{\rho})_{\bC}$ on $V_{\rho}$ is equivalent to
	$\rho \circ p_1$, where
	\[
		G'(A) \times G'(A_{\rho}) \supset G'(A \times A_{\rho}) \stackrel{p_1}{\longrightarrow} G'(A)
	\]
	is the projection to the first factor (see Remark \ref{productremark}),
	\item
	for each $\sigma \in \Aut (\bC)$, the Galois conjugate $\sigma (V_{\rho})$
	contains a nonzero $(c_{\rho}, 0)$-form.
	(The action of $\Aut (\bC)$ on $H^{c_{\rho}} (A_{\rho}, \bC) = H^{c_{\rho}} (A_{\rho}, \bQ) \otimes \bC$
	is through the second factor.)
\end{itemize}
\end{definition}

\begin{remark}
We will see in Theorem~\ref{semi} below that if $A$ is semidominated by~$\sX$, then
$A$ is dominated by abelian varieties.
However, an abelian variety may be dominated by abelian varieties without being semidominated
by any set of abelian varieties.
For example, let $A$ be a simple abelian variety of type~III which satisfies the hypotheses of
\cite{type3ghc}*{Theorem~4.1, p.~922}, and assume further that $D(A)$ has center $\bQ$, and
$m = \dim_{D(A)} H^1(A, \bQ) \geq 5$.
Then $A$ is dominated by the set of powers of itself \cite{type3ghc}*{Theorem~4.1, p.~922}, but we claim that it is not semidominated by any set of abelian varieties.
Suppose otherwise.
Then by the arguments in the proof of
\cite{type3ghc}*{Theorem~5.5, p.~926}, $A$ is semidominated by the set of powers of $A$.
Now let $V \subset H^m(A, \bQ)$ be an irreducible Hodge structure containing the representation $U_1 = U_{1,\alpha}$ defined in the proof of \cite{type3ghc}*{Theorem~4.1, p.~922}.
Then $V$ is fully twisted.
Let $V_1 \subset H^c(A^n, \bC)$ be the $G(A)_{\bC}$-module associated to $U_1$ by Definition~\ref{semidomination}. 
By \cite{type3ghc}*{Lemma~3.3.1, p.~921}, we must have $c=m$.
Then the smallest Hodge structure $\tV$ containing $V_1$ is isomorphic to $V$ (see the proof of Theorem~\ref{semi}).
This implies that every Galois conjugate of $U_1$ contains an $(m,0)$-form, which is seen to be false in the proof of \cite{type3ghc}*{Theorem~4.1, p.~922}.
\end{remark}

\begin{lemma}
\label{weakproduct}
If $A$ is semidominated by $\sX$ and $B$ is semidominated by $\sY$, and if
$G'(A \times B) = G'(A) \times G'(B)$, then $A \times B$ is semidominated by
\[
	\sX \cdot \sY = \{ X \times Y \mid X \in \sX,\ Y \in \sY \}.
\]
\end{lemma}

\begin{proof}
Any irreducible representation of $G'(A \times B)_{\bC}$ is of the form $\rho \otimes \tau$,
where $\rho$ is an irreducible representation of $G'(A)_{\bC}$ and
$\tau$ is an irreducible representation of $G'(B)_{\bC}$.
Let
\[
	V_{\rho \otimes \tau} =
		\begin{cases}
			V_{\rho} \otimes V_{\tau} &\text{if both $\rho$ and $\tau$ are nontrivial;} \\
			V_{\rho} &\text{if $\rho$ is nontrivial but $\tau$ is trivial;}\\
			V_{\tau} &\text{if $\tau$ is nontrivial but $\rho$ is trivial.}
		\end{cases}
\]
Then each Galois conjugate of $V_{\rho \otimes \tau}$ contains a $(c_{\rho \otimes \tau},0)$-form for some $c_{\rho \otimes \tau}$.
\end{proof}

\begin{theorem}
\label{semi}
Let $A$ be an abelian variety semidominated by $\sX$.
Then $A$ is dominated by the set of abelian varieties of the form $B \times C$,
where $B \in \sX$, and $C$ is of CM-type.
\end{theorem}

\begin{proof}
Let $V$ be any irreducible Hodge structure in the cohomology of $A$.
If $G'(A)$ acts trivially on $V$, then $V$ is of CM-type, so by
\cite{Abdulali2005}*{Theorem 3, p.~159} there exists an abelian variety $C$ of CM-type,
and a fully twisted Hodge structure $V'$ in the cohomology of $C$,
such that $V'$ is isomorphic to a Tate twist of $V$.
Otherwise, let $V_0$ be any irreducible $G(A)_{\bC}$-submodule of $V_{\bC}$;
it is necessarily irreducible as a $G'(A)_{\bC}$-module.
Since $A$ is semidominated by~$\sX$, there exists $B \in \sX$, and
an irreducible $G(B)_{\bC}$-submodule $V_{\rho}$ of $H^{c_{\rho}}(B, \bC)$
satisfying the conditions of Definition~\ref{semidomination}.

Let $T$ be the torus $G(A \times B)/G'(A \times B)$,
and let $T \to GL(W_0)$ be any faithful representation of $T$.
Since every representation of $G(A \times B)$ occurs in the tensor algebra of
$H^1(A \times B, \bQ)$ \cite{900}*{Proposition~3.1(a), p.~40}, there exists
a Hodge structure $W$  in the cohomology of some power of $A \times B$ such that
the action of $G(A \times B)$ on $W$ is equivalent to the representation
$G(A \times B) \to T \to GL(W_0)$.
Then $G(W) = T$, so, $W$ is of CM-type.

$V_0$ is equivalent to $V_{\rho}$ as a $G'(A \times B)_{\bC}$-module.
Therefore $V_0$ is equivalent to  $V_{\rho} \otimes \chi$ as a $G(A \times B)_{\bC}$-module,
where $\chi$ is a character of $T_{\bC}$.
The character $\chi$ occurs in the tensor algebra of $W$.
Let $W_{\chi}$ be its representation space, and let $Z$ be an irreducible Hodge structure
in the tensor algebra of $W$ such that $Z_{\bC}$ contains $W_{\chi}$.

By the main theorem of \cite{Abdulali2005}, there exists an abelian variety $C$ of CM-type and
an irreducible Hodge structure $Z' \subset H^c (C, \bQ)$ such that $Z'$ is isomorphic to
a Tate twist $Z(w)$ of $Z$, and, $Z'$ is fully twisted.
Let $\varphi \colon Z(w) \to Z'$ be an isomorphism of Hodge structures, and
let $Z'_{\chi} = \varphi(W_{\chi})$.
Then there exists $\sigma \in \Aut(\bC)$ such that $\sigma (Z'_{\chi})$ contains
a nonzero $(c,0)$-form.
Let
\[
	U' = V_{\rho} \otimes Z'_{\chi} \subset H^{c_{\rho}+c} (B \times C, \bC).
\]
Let $\tV'$  be the smallest Hodge substructure of $H^{c_{\rho}+c} (B \times C, \bQ)$
such that $\tV'_{\bC}$ contains $U'$.
Then $\tV'$ is fully twisted because $\sigma (U')$ contains a nonzero $(c_{\rho}+c, 0)$-form.
Recall that $V$ is an irreducible $G(A)$-module, and note that $\tV'$ is
a primary $G(B \times C)$-module, i.e., all of its irreducible submodules are mutually equivalent.
Let $G = G(A \times B \times C)$.
We have a $G_{\bC}$-isomorphism between $V_0$ and $U'$,
so $\hom_{G_{\bC}} (V_{\bC},\tV'_{\bC})$ is nontrivial.
Since $G(\bQ)$ is Zariski-dense in $G(\bC)$ (Rosenlicht \cite{Rosenlicht}*{Corollary, p.~44}),
we have
\[
	\hom_{G(\bC)} (V_{\bC},\tV'_{\bC}) =
	\hom_{G(\bQ)} (V_{\bC},\tV'_{\bC}) =
	\hom_G (V, \tV') \otimes \bC.
\]
Therefore $\hom_G (V, \tV')$ is nontrivial, and $V$ is isomorphic to a Tate twist of
an irreducible Hodge structure $V'$ contained in $\tV'$.
\end{proof}

\section{Abelian varieties associated to $SU(p,1)$}

Let $G$ be a simple algebraic group over $\bQ$ of hermitian type such that
$G(\bR) = \prod_{\alpha \in S} SU(p_{\alpha}, q_{\alpha})$.
In his classification of equivariant holomorphic maps, Satake found that in addition to the standard representation (and its contragredient), there are nonstandard representations of $G$ satisfying the $H_1$-condition \eqref{h1} if and only if one of $p_{\alpha}$, $q_{\alpha}$ equals $1$ for each $\alpha$.
We summarize his results for the groups of interest to us in the following theorem.

\begin{theorem}[Satake \citelist{\cite{Satake1965}*{\S 3.2, pp.~447--449} \cite{Satake1967}*{\S 8.3 (I$'$), pp.~268--269} \cite{Satakebook}*{\S IV.5, Exercise 1, p.~188}}]
\label{satake}
Let $G$ be a $\bQ$-simple algebraic group of hermitian type such that each simple factor of $G(\bR)$ is isomorphic to
$SU(p,1)$ for some $p>1$.
Let $m = p+ 1$. Then
\begin{enumerate}
\item
There exists a totally real number field $F$ and an absolutely simple algebraic group $\tG$ over $F$ such that
$G = \Res_{F/{\bQ}} \tG$.
Let $S$ be the set of embeddings of $F$ into $\bR$.
Then $G(\bR) = \prod_{\alpha \in S} G_{\alpha}$, where
$G_{\alpha} = \tG \otimes_{F,\alpha} \bR \cong SU(p,1)$.
\item \label{maintheoremmainpart}
For each $k = 1$, \dots, $p$, there exists a homomorphism of $\bQ$-algebraic groups
\begin{equation}
\label{satakemap}
	\rho_k \colon G \to Sp(V_k, \beta_k),
\end{equation}
where $\beta_k$ is a nondegenerate alternating form on a finite-dimensional
$\bQ$-vector space $V_k$, $\rho_k$ satisfies the $H_1$-condition \eqref{h1}, and, $\rho_k$
is equivalent over $\bR$ to
$\bigoplus_{\alpha \in S} \left(\bigwedge^k \oplus \bigwedge^{m-k}\right) \circ P_{\alpha}$,
where $P_{\alpha} \colon G(\bR) \to G_{\alpha}$ is the projection.
\item \label{converse}
If $\rho$ is any symplectic representation of $G$ which satisfies the $H_1$-condition, then $\rho$ is equivalent to a direct sum $\bigoplus_{i=1}^p m_i \rho_i \oplus (\textup{triv})$, where $(\textup{triv})$ denotes a trivial representation.
\end{enumerate}
\end{theorem}

\begin{remark}
\label{satakeremark}
We make several remarks about Theorem~\ref{satake} and its proof.
Let $\sA_k \to \sV$ be the Kuga fiber variety defined by
the symplectic representation \eqref{satakemap}.
\begin{enumerate}
\item
$\sA_1 \to \sV$ is a \pel-family.
\item
\label{remarkdual}
	$\rho_k$ and $\rho_{m-k}$ are equivalent.
\item
\label{remarkhodgeitem}
	Since $G(\bR)$ has no compact factors, it follows from
	\cite{Abdulali1994}*{Remark~3.5, p. ~213} that each  $\sA_k \to \sV$ is a Hodge family.
\item
	$\rho_k$ factors as
	\[
		G \stackrel{\tilde{\rho}_k}{\longrightarrow} G_k
		\stackrel{\tilde{\varphi}_k}{\longrightarrow} Sp(V_k, \beta_k),
	\]
	where $G_k = \Res_{F/\bQ} \widetilde{G}_k$ for a semisimple algebraic group
	$\widetilde{G}_k$ over $F$.
	We have $G_k(\bR) =  \prod_{\alpha \in S} G_{k,\alpha}$,
	where $G_{k,\alpha} \cong SU(p', q')$.
	Here, $p' = \binom{p}{k}$ and $q' = \binom{p}{k-1}$.
\item
	$\tilde{\rho}_{k,\bR} \colon G(\bR) \to G_k (\bR)$ can be described as follows:
	for $g = (g_{\alpha}) \in G(\bR) = \prod_{\alpha \in S} G_{\alpha}$, we have
	\[
		\tilde{\rho}_{k,\bR} (g) = \tilde{\rho}_{k,\alpha} (g_{\alpha}) \in
		G_k(\bR) =  \prod_{\alpha \in S} G_{k,\alpha},
	\]
	where
	$\tilde{\rho}_{k,\alpha} \colon G_{\alpha} = SU(p,1) \to G_{k,\alpha} = SU(p',q')$
	is equivalent to the representation on the $k$-th exterior power.
\item
\label{remarkdecomp}
	By part \ref{maintheoremmainpart} of Theorem~\ref{satake}, we can write
	\begin{equation}
	\label{decomp}
		V_{k,\bR} = \bigoplus_{\alpha \in S} V_{k,\alpha},
	\end{equation}
	where $G_{\alpha}$ acts trivially on $V_{k,\alpha'}$ unless $\alpha = \alpha'$, and
	$G_{\alpha,\bC}$ acts as $\bigwedge^k \oplus \bigwedge^{m-k}$ on $V_{k,\alpha,\bC}$.
\item
	$G_{k,\alpha}$ acts trivially on $V_{k,\alpha'}$ unless $\alpha = \alpha'$, and, 
	the action of $G_{k,\alpha,\bC}$ on $V_{k,\alpha,\bC}$
	is equivalent to the direct sum of the standard representation and its contragredient.
\item
\label{remarktype}
	Let $A_k$ be a general fiber of $\sA_k \to \sV$.
	If $k \neq \frac{m}{2}$, then $A_k$ is of type IV.
	If $k = \frac{m}{2}$, then $A_k$ is not of type IV, and both $G(A_k)$ and $L(A_k)$ are semisimple.
\end{enumerate}
\end{remark}

\begin{theorem}
\label{maintheorem}
Let $A$ be an abelian variety such that each simple factor of $G'(A)(\bR)$ is isomorphic to $SU(p,1)$ for some $p > 1$.
Then $A$ is dominated by abelian varieties.
\end{theorem}

\begin{proof}
Thanks to Theorem~\ref{semi}, it  suffices to show that $A$ is semidominated by some set of abelian varieties.
By Lemmas~\ref{groupreduction} and~\ref{weakproduct}, we may assume that $G'(A)$ is simple.
We are thus in the situation of Theorem~\ref{satake} with $G = G'(A)$.

Let $\sA \to \sV$ be the Hodge family of $A$, and let $P \in \sV$ be such that $\sA_P = A$.
For each $k = 1$, \dots, $p$, there is a Kuga fiber variety ${\sA}_k \to \sV$
defined by the symplectic representation $\rho_k$ \eqref{satakemap}.
Let $A_k = (\sA_k)_P$.
We will show that $A$ is  semidominated by the set 
\[
	\sX = \left\{ A_1^{n_1} \times \dots \times  A_p^{n_p} \mid n_i \geq 0 \right\}.
\]

Assume without loss of generality that $A$ has no factors of CM-type.
Then the $H_1$-representation defining $\sA \to \sV$ does not contain the trivial representation,
so by part \ref{converse} of Theorem~\ref{satake}, it is equivalent to $m_1\rho_1 + \dots + m_p\rho_p$,
for some nonnegative integers $m_1, \dots, m_p$.
Thus $\sA$ is isogenous to $\sA_1^{m_1} \times_{\sV} \dots \times_{\sV} \sA_p^{m_p}$.
It follows that $A$ is isogenous to $A_1^{m_1} \times \dots \times A_p^{m_p}$.
From now on we assume without loss of generality that
\[
	A = A_1^{m_1} \times \dots \times A_p^{m_p},
\]
and $m_1 > 0$.
In particular, $A \in \sX$.

We claim that $G'(A \times X) = G'(A)$ for any $X \in \sX$.
More precisely, we mean that the projection $G(A) \times G(X) \to G(A)$, restricted to the subgroup
$G'(A \times X)$ of $G(A) \times G(X)$ induces an isomorphism of $G'(A \times X)$ with $G'(A)$.
To see this, let $X = A_1^{n_1} \times \dots \times  A_p^{n_p}$.
Then
\[
	A \times X = A_1^{m_1+n_1} \times A_2^{m_2+n_2} \times \dots \times  A_p^{m_p+n_p}.
\]
Thus $A \times X$ is the fiber over $P$ of the Kuga fiber variety
\[
	\sB = \sA_1^{m_1+n_1} \times_{\sV} \sA_2^{m_2+n_2}  \times_{\sV} \dots \times_{\sV}   \sA_p^{m_p+n_p}
	\to \sV
\]
defined by the symplectic representation
\[
	\rho = (m_1+n_1)\rho_1 \oplus (m_2+n_2) \rho_2 \oplus \cdots \oplus (m_p+n_p) \rho_p
\]
 of $G'(A)$.
This is a Hodge family (see Remark \ref{satakeremark}.\ref{remarkhodgeitem}).
Hence $G'(\sB_Q) \subset G'(A)$ for all $Q \in \sV$.
But $G' (A)$ is a quotient of $G'(A \times X) = G'(\sB_P)$.
Therefore $G'(A \times X) = G'(A)$.

We can write $L(A_1)_{\bR} = \prod_{\alpha \in S} L_{1,\alpha}$ and
$V_{1,\bR} = \bigoplus_{\alpha \in S} V_{1,\alpha}$ where
$L_{1,\alpha} \cong U(p, 1)$ acts trivially on $V_{1,\alpha'}$ unless
$\alpha = \alpha'$, and,  $L_{1,\alpha,\bC} \cong GL_m(\bC)$ acts on $V_{1,\alpha,\bC}$
as the direct sum of the standard representation and its contragredient (Murty \cite{Murty}).
As explained in \cite{Abdulali1997}*{p.~351},
$V_{1,\alpha, \bC} = Y_{\alpha} \oplus \overline{Y}_{\alpha}$, where $Y_{\alpha}$ and
its complex conjugate $\overline{Y}_{\alpha}$ are  $L_{1,\alpha, \bC}$-modules,
$GL_m(\bC)$ acts on $Y_{\alpha}$ as the standard representation,
and on $\overline{Y}_{\alpha}$ as the contragredient.
$Y_{\alpha}$ is the direct sum of a $p$-dimensional space of $(1,0)$-forms and
a $1$-dimensional space of $(0,1)$-forms.
$\overline{Y}_{\alpha}$ is the direct sum of a $1$-dimensional space of $(1,0)$-forms and a
$p$-dimensional space of $(0,1)$-forms.
Choose a basis $\left\{ u_1, \dots, u_m \right\}$ of $Y_{\alpha}$ such that
$u_1$, \dots, $u_{p}$ are $(1,0)$-forms and $u_m$ is a $(0,1)$-form.
Then $\left\{ \overline{u}_1, \dots, \overline{u}_m \right\}$ is a basis of $\overline{Y}_{\alpha}$.

There is a natural action (from the left) of $\Aut (\bC)$ on $S$, the set of embeddings of $F$ into $\bR$.
For $\sigma \in \Aut (\bC)$ and $\alpha \in S$ we have
$\sigma \left\{ Y_{\alpha}, \overline{Y}_{\alpha} \right\} = \left\{ Y_{\sigma\alpha}, \overline{Y}_{\sigma\alpha} \right\}$.
Thus the set $\bigcup_{\alpha \in S} \left\{ Y_{\alpha}, \overline{Y}_{\alpha} \right\}$
is invariant under the action of $\Aut (\bC)$, and every Galois conjugate of $Y_{\alpha}$
contains a nonzero $(1,0)$-form.

Let $\mu_1$, \dots, $\mu_{m-1}$ be the fundamental weights of $SL_{m}(\bC)$, i.e., $\mu_k$
is the highest weight of the representation $\bigwedge^k (\St)$, where $(\St)$ denotes
the standard representation of $SL_{m}(\bC)$ on $\bC^m$.

For $1 \leq k \leq p$ and $\alpha \in S$, let $V_{k,\alpha}$ be as in Remark \ref{satakeremark}.\ref{remarkdecomp}.
Then $V_{k,\alpha,\bC} = Y_{k,\alpha} \oplus \overline{Y}_{k,\alpha}$, where
$Y_{k,\alpha} = \bigwedge^k Y_{\alpha}$ and
$\overline{Y}_{k,\alpha} = \bigwedge^k \overline{Y}_{\alpha}$ are  $GL_m(\bC)$-modules,
$GL_m(\bC)$ acting on $Y_{k,\alpha}$ as $\bigwedge^k (\St)$, and on $\overline{Y}_{k,\alpha}$
as the contragredient.

$Y_{k,\alpha}$ is the direct sum of a $\binom{p}{k}$-dimensional space of $(1,0)$-forms and
a $\binom{p}{k-1}$-dimensional space of $(0,1)$-forms.
In particular, the highest weight vector in $Y_{k,\alpha}$ is
\[
	w_k = u_1 \wedge \cdots \wedge u_k
\]
which is a $(1,0)$-form.
Define $\gamma \in GL_m(\bC)$ by $\gamma (u_i) = u_{m+1-i}$.
Then
\[
	w'_k = \gamma (w_k) = u_{m} \wedge \cdots \wedge u_{m+1-k}
\]
is a $(0,1)$-form.
$\overline{Y}_{k,\alpha}$ is the direct sum of a $\binom{p}{k-1}$-dimensional space of
$(1,0)$-forms and a $\binom{p}{k}$-dimensional space of $(0,1)$-forms.
Observe that the set $\bigcup_{\alpha \in S} \left\{ Y_{k,\alpha}, \overline{Y}_{k,\alpha} \right\}$
is invariant under the action of $\Aut (\bC)$, so every Galois conjugate of $Y_{k,\alpha}$
contains a nonzero $(1,0)$-form.

Let $j$ be a positive integer.
Then $S^j Y_{k,\alpha}$, the space of symmetric tensors on $Y_{k,\alpha}$,
is a representation of $SL_m(\bC)$ with highest weight $j\mu_k$,
and highest weight vector $(w_k)^j$.
Let $V^j_{k,\alpha} \subset H^j(A_k^j, \bC)$ be the $SL_m(\bC)$-module generated by $(w_k)^j$.
The highest weight vector in $V^j_{k,\alpha}$ is $(w_k)^j$ which is a $(j,0)$-form.
For $z \in \bC^{\times}$, the scalar matrix $z \cdot I$ acts on $S^j Y_{k,\alpha}$ as multipliction by $z^j$.
It follows that any $SL_m(\bC)$-submodule of $S^j Y_{k,\alpha}$,
in particular $V^j_{k,\alpha}$, is a $GL_m(\bC)$-submodule.
Therefore $V^j_{k,\alpha}$ contains the $(0,j)$-form $\gamma (w_k)^j = (w'_k)^j$.

For $\sigma \in \Aut (\bC)$ and $\alpha \in S$ we have
$\sigma \{ V^j_{k,\alpha}, \overline{V}^j_{k,\alpha} \} = 
\{ V^j_{k,\sigma\alpha}, \overline{V}^j_{k,\sigma\alpha} \}.$
Hence the set
\[
	\left\{ \, V^j_{k,\alpha} \mid \alpha \in S, \ 1 \leq k \leq p, \ j > 0 \, \right\} \cup
	\left\{ \, \overline{V}^j_{k,\alpha} \mid \alpha \in S, \ 1 \leq k \leq p, \ j > 0 \, \right\}
\]
is invariant under the action of $\Aut (\bC)$, and every Galois conjugate of $V^j_{k,\alpha}$
contains a nonzero $(j,0)$-form.

Any irreducible representation $\pi$ of $SL_{m}(\bC)$ has highest weight
\[
	\mu = a_1\mu_1 + \dots + a_p\mu_p
\]
where the $a_j$ are nonnegative integers.
Let $a_{\mu} = \sum_{k=1}^p a_k.$
Then the representation
\[
	\bigotimes_{k=1}^p V_{k,\alpha}^{a_k} \subset
	H^{a_\mu}(A_1^{a_1} \times \cdots \times  A_p^{a_p}, \bC)
\]
has highest weight $\mu$.
The vector $v_{\mu} = \bigotimes_{k=1}^p (w_k)^{a_k}$
generates an irreducible submodule $V_{\mu}^{\alpha}$ with highest weight $\mu$.
Note that a scalar matrix $z \cdot I$ acts on $\bigotimes_{k=1}^p V_{k,\alpha}^{a_k}$
as multiplication by $\prod_{k=1}^p z^{a_k}$, so $V_{\mu}^{\alpha}$ is a $GL_m(\bC)$-module.
Hence $V_{\mu}^{\alpha}$ contains both the $(a_{\mu},0)$-form  $v_{\mu}$ and
the $(0,a_{\mu})$-form $v'_{\mu} = \gamma (v_{\mu}) = \bigotimes_{k=1}^p (w'_k)^{a_k}$.
For $\sigma \in \Aut (\bC)$, we have
$\sigma (V_{\mu}^{\alpha}) \in \{ V_{\mu}^{\sigma\alpha}, V_{\overline{\mu}}^{\sigma\alpha} \}$,
where $\overline{\mu} = a_p \mu_1 + \dots + a_1 \mu_p$.
Thus the set
\[
	\bigg\{ \, V_{\mu}^{\alpha}\ \big\vert \ \alpha \in S, \
	\mu = \sum_{i=1}^p a_i \mu_i, a_i \geq 0 \, \bigg\}
\]
is invariant under the action of $\Aut (\bC)$, and every Galois conjugate of
$V_{\mu}^{\alpha}$ contains a nonzero $(a_{\mu},0)$-form.

Any irreducible representation $\rho$ of $G'(A)_{\bC} = \prod_{\alpha \in S} G'(A)_{\alpha, \bC}$
is of the form $\rho = \bigotimes_{\alpha \in S} \pi_{\alpha}$, where $ \pi_{\alpha}$
is an irreducible representation of $G'(A)_{\alpha, \bC} \cong SL_m(\bC)$.
Let
\[
	V_{\rho} = \bigotimes_{\alpha \in S} V_{\mu_\alpha}^{\alpha},
\]
where $\mu_{\alpha}$ is the highest weight of $\pi_{\alpha}$.
Then $V_{\rho}$ is an irreducible $G'(A)_{\bC}$-submodule of $H^c(A_{\rho}, \bC)$,
for some $A_{\rho} \in \sX$ and some positive integer $c$, on which $G'(A)_{\bC}$ acts as $\rho$.
Since every Galois conjugate of each $V_{\mu_\alpha}^{\alpha}$
contains a nonzero $(a_{\mu_{\alpha}}, 0)$ form,
it follows that every Galois conjugate of $V_{\rho}$ contains a nonzero $(c,0)$-form.

To complete the proof of the theorem we need to show that
the $G'(A_{\rho})_{\bC}$-module $V_\rho$ is actually a $G(A_{\rho})_{\bC}$-module.
By Proposition~\ref{center}, it is sufficient to show that
$V_{\rho}$ is a $Z\big(L(A_{\rho})\big)_{\bC}$-module.
Since $A_k$ and $A_{m-k}$ are isomorphic (Remark~\ref{satakeremark}.\ref{remarkdual}), we may write
\[
	A_{\rho} = \prod_{2k \leq m}  A_k^{n_k}.
\]
Assume, without loss of generality, that each $n_k > 0.$ Then
\[
	L(A_{\rho}) = \prod_{2k \leq m} L(A_k).
\]
If $k = \frac{m}{2}$, then $G(A_k)$ and $L(A_k)$ are semisimple (Remark \ref{satakeremark}.\ref{remarktype}), so it is sufficient to show that
$V_{\rho}$ is a $Z$-module, where
\[
	Z = \prod_{1 \leq k < \frac{m}{2}} Z\big(L(A_k)\big)_{\bC}.
\]
Now, for $k \neq \frac{m}{2}$,
\[
	Z\big(L(A_k)\big)_{\bC} = \prod_{\alpha \in S} Z_{k,\alpha}
\]
where each $Z_{k,\alpha} \cong \bC^{\times}$, and in the decomposition \eqref{decomp},
$Z_{k,\alpha}$ acts trivially on $V_{k,\alpha'}$ unless $\alpha = \alpha'$.
A scalar $z_{\alpha} \in Z_{k.\alpha}$ acts on $Y_{k,\alpha}$ as multiplication by $z_{\alpha}$,
and on $\overline{Y}_{k,\alpha}$ as multiplication by $\overline{z}_{\alpha}$;
thus $Y_{k,\alpha}$ and $\overline{Y}_{k,\alpha}$ are both $Z_{k,\alpha}$-modules.
The scalar $z_{\alpha}$ acts as multiplication by $z^j_{\alpha}$ on $V^j_{k,\alpha}$;
so $V^j_{k,\alpha}$ is a $Z_{k,\alpha}$-module.
Now let
\[
	z = (z_{k,\alpha}) \in Z\big(L(A_{\rho})\big)_{\bC} =
	\prod_{1 \leq k < \frac{m}{2}} \ \prod_{\alpha \in S} Z_{k,\alpha}.
\]
For a given $\alpha \in S$, $z$ acts on $\bigotimes_{k=1}^p V_{k,\alpha}^{a_k}$ as multiplication by
$\prod_{2k<m} z_{k,\alpha}^{a_k}$.
Hence any subspace of it (in particular, $V^{\alpha}_{\mu_{\alpha}}$) is a
$Z_{\bC}$-module.
It follows that $V_{\rho} = \bigotimes_{\alpha \in S} V_{\mu_\alpha}^{\alpha}$ is a
$Z_{\bC}$-module.
\end{proof}

\begin{remark}
The assumption that $p>1$ in Theorem~\ref{maintheorem} is only for convenience.
If $p=0$, then a result of Shimura \cite{Shimura1963b}*{Proposition~14, p.~176}
implies that $A$ is of CM-type, so the theorem is a special case of \cite{Abdulali2005}*{Theorem~4, p.~159}.
If $p=1$, then $SU(p,1) \cong Sp_2(\bR)$, so the theorem is a special case of \cite{Abdulali1997}*{Theorem~5.1, p.~348}.
\end{remark}

\end{document}